\documentclass[12pt]{article}

\usepackage{amsmath,amssymb,amsthm,latexsym,amscd,stmaryrd,mathrsfs,longtable,color,empheq}

\newcommand{\db}[1]{(\!({#1})\!)}

\numberwithin{equation}{section}

\newcommand{\N}{{\mathbb N}}
\newcommand{\Z}{{\mathbb Z}}

\newcommand{\C}{{\mathbb C}}

\newcommand{\wi}{i}
\newcommand{\wj}{j}
\newcommand{\wk}{k}
\newcommand{\wl}{l}

\newcommand{\wx}{x}

\newcommand{\wor}{r}
\newcommand{\ws}{s}
\newcommand{\wm}{m}
\newcommand{\wn}{n}
\renewcommand{\wp}{p}
\newcommand{\hei}{{\mathfrak h}}
\newcommand{\module}{M}

\renewcommand{\wr}{r}

\newcommand{\sN}{{\mathscr N}}

\newcommand{\sU}{{\mathscr U}}
\newcommand{\sW}{{\mathscr W}}

\newcommand{\wq}{q}

\newcommand{\nor}{\begin{subarray}{c}\circ\\\circ\end{subarray}}

\newcommand{\fg}{{\mathfrak g}}
\newcommand{\fh}{{\mathfrak h}}

\newcommand{\ul}[1]{{#1}}
\newcommand{\lw}{u}
\newcommand{\wv}{v}
\newcommand{\ulw}{v}

\newcommand{\vac}{{\mathbf 1}}
\newcommand{\wot}{t}

\newcommand{\sla}[1]{\mathcal M({#1})}

\DeclareMathOperator{\Res}{Res}

\DeclareMathOperator{\Span}{Span} 
\DeclareMathOperator{\wt}{wt}

\newtheorem{lemma}{Lemma}[section]
\newtheorem{theorem}[lemma]{Theorem}

\newtheorem{corollary}[lemma]{Corollary}
\theoremstyle{definition}

\newtheorem{remark}[lemma]{Remark}
\title{Simple weak modules for some subalgebras of the Heisenberg vertex algebra
and Whittaker vectors}
\author{Kenichiro Tanabe\footnote{This research was partially supported by JSPS Grant-in-Aid for Scientific Research No. 15K04770.}
\\\\
Department of Mathematics\\
Hokkaido University\\
Kita 10, Nishi 8, Kita-Ku, Sapporo, Hokkaido, 060-0810\\
Japan\\\\
ktanabe@math.sci.hokudai.ac.jp}
\date{}
\begin{document}
\maketitle

\numberwithin{equation}{section}

\abstract{
Let $\sla{p}$ $(p=2,3,\ldots)$ be the singlet vertex operator algebra
and $\omega$ its conformal vector. 
We classify the simple weak $\sla{p}$-modules with a non-zero element $ u$ such that
for some integer $s\geq 2$,  $\omega_i u\in{\mathbb C} u$ ($i=\lfloor s/2\rfloor+1,\lfloor s/2\rfloor+2,\ldots,s-1$), 
$\omega_{s} u\in{\mathbb C}^{\times} u$, and $\omega_i u=0$ for all $i>s$. 
}



\section{\label{section:introduction}Introduction}
The singlet vertex operator algebra $\sla{p}$ $(p=2,3,\ldots)$ is
a subalgebra of the Heisenberg vertex algebra $M(1)$ of rank $1$ with two generators,
and 
is known as $\sW(2,2p-1)$-algebra in the physics literature (cf. \cite[Section V]{BCNM}, \cite{Ka}).
Since $\sla{p}$ admits infinitely many non-isomorphic simple (irreducible) modules,
it is non-$C_2$-cofinite and non-rational.
The  vertex operator algebra $\sla{p}$ has been studied from various perspectives (cf. \cite{A},\cite{AM1},\cite{AM2},\cite{BM},\cite{CM},\cite{CMP},\cite{W1},\cite{W2}),
partly because of its connection to the triplet vertex algebra.
For the representation of $\sla{p}$,
Adamovi\'c \cite{A} classifies the simple $\sla{p}$-modules.
The notation $\overline{M(1)}$ is used there instead of $\sla{p}$.
So, one of the next task is to investigate  simple weak $\sla{p}$-modules.
The purpose of this paper is 
to classify the simple weak $\sla{p}$-modules with  a Whittaker vector for $\omega$ in the sense of 
\cite[(1.1)]{T1}.

Whittaker modules (Whittaker vectors) are non-weight modules defined over various Lie algebras, first appeared in \cite{AP} for $sl_{2}$.
Whittaker modules  for any 
finite dimensional complex simple Lie algebra are systematically studied in \cite{K1} 
and applied to the study of the Toda lattice in \cite{K2}.
Results in \cite{K2} are generalized to
affine Lie algebras or quantum groups in \cite{E} and \cite{S}.
Whittaker modules are also studied for the Virasoro algebra in \cite{OW}, \cite{LGZ} and \cite{FJK}, and
for the affine Kac--Moody algebra $A_1^{(1)}$ in \cite{ALZ}.
Whittaker modules for the Virasoro algebra also appear in the study of 
two-dimensional conformal field theory in physics(cf.\cite{G}).
Whittaker modules for a general vertex operator algebra $V$ are not defined, however, we note
that for the conformal vector (the Virasoro element) $\omega$ of $V$,
$\omega_{n+1}=L(n),n\in\Z$ satisfy the Virasoro algebra relations (cf. \cite[(1.3.4)]{LL}).
Thus, based on the definition of Whittaker vectors for the Virasoro algebra in \cite{OW} and \cite{LGZ},
I introduced the following notion in \cite{T1}:
for a weak $V$-module $M$, a non-zero element $\lw$ of $M$ is called a {\it Whittaker vector for $\omega$}  if 
there exists an integer $\ws$ with $\ws\geq 2$ and 
$\ul{\lambda}=
(\lambda_{\lfloor \ws/2\rfloor+1},\lambda_{\lfloor \ws/2\rfloor+2},\ldots,\lambda_{\ws})\in \C^{\ws-\lfloor \ws/2\rfloor}$
with $\lambda_{\ws}\neq 0$ such that
\begin{align}
\omega_{\wi}\lw=
\left\{
\begin{array}{ll}
\lambda_{\wi}\lw,& \wi=\lfloor\ws/2\rfloor +1,\lfloor\ws/2\rfloor +2,\ldots,\ws.\\
0,&\mbox{if  }i>\ws,
\end{array}\right.
\end{align}
where $\lfloor\ws/2\rfloor=\max\{i\in\Z\ |\ i\leq \ws/2\}$.
We call $\ul{\lambda}$ the {\it type} of $\lw$. 
We note that if we regard a weak $V$-module $M$ as a module for the Virasoro algebra,
then the Whittaker vectors in $M$ for $\omega$ coincide with the 
Whittaker vectors in $M$ for the Virasoro algebra.
It is well known that for each non-negative integer $r$ and
$\zeta=(\zeta_0,\ldots,\zeta_{r})\in\C^{r}\times \C^{\times}$,
the Heisenberg vertex algebra $M(1)$ has a simple weak module $M(1,\zeta)$ (See \eqref{eq:untwist-induced}),
which is also a Whittaker module for the Heisenberg algebra.
Let $\omega$ be the conformal vector of $\sla{p}$. 
As we shall see later in Corollary \ref{corollray:Mzeta},
for every $\zeta\in \C^{r}\times \C^{\times}$
the weak $\sla{p}$-module $M(1,\zeta)$ 
is simple and contains a Whittaker vector $u_{\zeta}$ of
type 
\begin{align}
\label{eq:whittaker-vector-type-intro}
\big(\dfrac{1}{4p}\sum_{\begin{subarray}{l}j,k\in\N\\j+k=i-1\end{subarray}}
\zeta_{j}\zeta_{k}-\delta_{i,r+1}\dfrac{(p-1)(r+1)}{2p}\zeta_{r}\big)_{i=r+1}^{2r+1}
\end{align}
for $\omega$.
The following is the main result of this paper, which implies that
the converse is true:
\begin{theorem}
\label{theorem:classification-module}
The set $\{M(1,\ul{\zeta})\ |\ r\in\Z_{\geq 1}, {\zeta}\in \C^{r}\times \C^{\times}\}$
is a complete set of representatives of equivalence classes of simple weak $\sla{p}$-modules with a Whittaker vector for $\omega$.
\end{theorem}
It is worth mentioning that the simple weak $M(1)^{+}$-modules 
with  a Whittaker vector for its conformal vector are already classified in \cite[Theorem 1.1]{T1}, where $M(1)^{+}$ is the subalgebra of $M(1)$ consisting of the fixed points of an automorphism of $M(1)$ of order $2$.
Applying a slight modification of the methods in \cite{T1} to  $\sla{p}$,
we can show Theorem \ref{theorem:classification-module}.
Let us explain the basic idea briefly.
It is shown in  \cite[Theorem 3.2]{A}  that $\sla{p}$ is generated by the conformal vector $\omega$ and homogeneous $H\in \sla{p}$ of weight $2p-1$.
Let $M$ be a weak $\sla{p}$-module generated by a Whittaker vector $u$ 
of type $\lambda=(\lambda_{i})_{i=\lfloor\ws/2\rfloor +1}^{\ws}$ for $\omega$.
Fist we find two relations for $\omega$ and $H$ in $\sla{p}$ (See \eqref{eq:relation-wH}
and \eqref{eqref:H-1H}). Using these relations,
we have that $\ws$ is an odd integer and
\begin{align}
H_{i}u&=0\mbox{ for all }i>\dfrac{(\ws+1)(2p-1)}{2}-1
\end{align}
and that $H_{(\ws+1)(2p-1)/2-1}u\in\C u$ if and only if 
$M$ is simple (See Lemma \ref{lemma:gen-omega} (4) and (5)).
In this case 
\begin{align}
\label{eq:actionHm}
H_{(\ws+1)(2p-1)/2-1}u
&=qu
\end{align}
where $q$ is a square root of $(4p)^{2p-1}\lambda_{\ws}^{2p-1}/((2p-1)!)^2$,
$M$ is generated by $u$ as a module for the Virasoro algebra associated to $\omega$, 
and the actions of $H$ on $M$ are uniquely
determined by the scalar $q$ and the actions of $\omega$ on $M$.
In particular, if $M$ is simple, then
$M$ is a Whittaker module for the Virasoro algebra in the sense of \cite{LGZ,OW}.
Since any Whittaker module for the Virasoro algebra
is uniquely determined by its type by \cite[Proposition 3.2]{OW} and \cite[Theorem 2.3]{FJK},
every simple weak $\sla{p}$-module with a Whittaker vector $u$ for $\omega$
is uniquely determined by the type of $u$ for $\omega$ and
$q\in \C$ in \eqref{eq:actionHm}.
Thus, the theorem follows from \eqref{eq:whittaker-vector-type-intro} and the computation 
of $H_{(\ws+1)(2p-1)/2-1}u_{\zeta}$ in $M(1,\zeta)$.

The organization of the paper is as follows. 
In Section \ref{section:preliminary} we recall some basic properties of
the Heisenberg vertex algebra $M(1)$, the singlet vertex algebra $\sla{p}$, and 
the weak $M(1)$-modules $M(1,\zeta)$ where $r\in\Z_{>0}$ and $\zeta\in \C^{r}\times\C^{\times}$.
We show that $M(1,\zeta)$, $r\in\Z_{>0}$ and $\zeta\in \C^{r}\times\C^{\times}$, are pairwise non-isomorphic simple weak $\sla{p}$-modules.
We also find two relations for $\omega$ and $H$ in $\sla{p}$.
In Section \ref{section:main} we give a proof
of Theorem \ref{theorem:classification-module}.

\section{\label{section:preliminary}Preliminary}
We assume that the reader is familiar with the basic knowledge on
vertex algebras as presented in \cite{B,FLM,LL}. 

Throughout this paper, $p$ is an integer with $p\geq 2$,
$\N$ denotes the set of all non-negative integers,
$\C^{\times}=\{z\in\C\ |\ z\neq 0\}$ and
$(V,Y,{\mathbf 1},\omega)$ is a vertex operator algebra.
Recall that $V$ is the underlying vector space, 
$Y(-,\wx)$ is the linear map from $V\otimes_{\C}V$ to $V\db{x}$,
${\mathbf 1}$ is the vacuum vector, and $\omega$ is the conformal vector.
A weak $V$-module $M$ (cf. \cite[p.157]{Li1}) is called {\it $\N$-graded}
if $M$ admits a decomposition $M=\oplus_{j=0}^{\infty}M(j)$ such that
\begin{align}
\label{eq:N-graded}
a_{k}M(j)\subset M(\wt a+j-k-1)
\end{align}
for homogeneous $a\in V$, $j\in \N$, and $k\in\Z$.
For a weak $V$-module $M$ and $U\subset M$, 
$\langle \omega\rangle U$ denotes the set of linear span of the following elements:
\begin{align}
\omega_{i_1}\cdots \omega_{i_n}u\ (n\in\N, i_1,\ldots,i_n\in\Z, u\in U).
\end{align}
When $U=\{u\}$, we simply write $\langle \omega\rangle u$ for $\langle \omega\rangle U$.
For $i\in\Z$, we define
\begin{align}
\Z_{< i}&=\{k\in\Z\ |\ k< i\}\mbox{ and }\Z_{> i}=\{k\in\Z\ |\ k> i\}.
\end{align}
We also define $\Z_{\leq i}$ and 
$\Z_{\geq i}$ similarly.

We recall the vertex operator algebra $M(1)$ associated to the Heisenberg algebra of rank $1$ and some weak $M(1)$-modules.
Let $\hei$ be a one dimensional vector space equipped with a nondegenerate symmetric bilinear form
$\langle -, -\rangle$.
Set a Lie algebra
\begin{align}
\hat{\hei}&=\hei\otimes \C[t,t^{-1}]\oplus \C K
\end{align} 
with the Lie bracket relations 
\begin{align}
[h\otimes t^{m},h^{\prime}\otimes t^{n}]&=m\langle h,h^{\prime}\rangle\delta_{m+n,0}K,&
[K,\hat{\hei}]&=0
\end{align}
for $h,h^{\prime}\in \hei$ and $m,n\in\Z$.
For $h\in \hei$ and $n\in\Z$, $h(n)$ denotes $h\otimes t^{n}\in\widehat{H}$. 
Set two Lie subalgebras of $\widehat{\hei}$:
\begin{align}
\widehat{\hei}_{{\geq 0}}&=\bigoplus_{n\geq 0}\hei \otimes t^n&\mbox{ and }&&
\widehat{\hei}_{<0}&=\bigoplus_{n\leq -1}\hei\otimes t^n.
\end{align}
We take $\alpha\in\hei$ such that 
\begin{align}
\langle \alpha,\alpha\rangle=2p.
\end{align}
For a non-negative integer $r$ and 
an $r+1$-tuple $\ul{\zeta}=(\zeta_0,\ldots,\zeta_{r})\in\C^{r+1}$,
$\C u_{\ul{\zeta}}$ denotes a one dimensional $\widehat{\hei}_{{\geq 0}}$-module uniquely determined
by
\begin{align}
\label{eq:alpha-uz-def}
\alpha(i)\cdot u_{\zeta}&
=\left\{
\begin{array}{ll}
\zeta_iu_{\zeta}&\mbox{ for }i=0,\ldots,r,\\
0&\mbox{ for }i>r
\end{array}
\right.&&\mbox{ and }&
K\cdot u_{\zeta}&=u_{\zeta}.
\end{align}
We take an $\widehat{\hei}$-module 
\begin{align}
\label{eq:untwist-induced}
M(1,\zeta)&=\sU (\widehat{\hei})\otimes_{\sU (\widehat{\hei}_{\geq 0})}\C u_{\zeta}
\cong \sU(\widehat{\fh}_{<0})\otimes_{\C}\C u_{\ul{\zeta}}
\end{align}
where $\sU(\fg)$ is the universal enveloping algebra of a Lie algebra $\fg$.
Then, $M(1)=M(1,(0))$ has a vertex algebra structure
and $M(1,\ul{\zeta})$ is a simple weak $M(1)$-module for any $\ul{\zeta}\in\C^{r+1}$.
The vertex operator algebra $M(1)$ is called the {vertex operator algebra associated to
 the Heisenberg algebra} $\oplus_{0\neq n\in\Z}\fh\otimes t^{n}\oplus \C K$. 
If $r=0$, then $M(1,(\lambda_0))$ is a simple $M(1)$-module. 
We define
\begin{align}
\omega&=\dfrac{\alpha(-1)^2\vac}{4p}+\dfrac{p-1}{2p}\alpha(-2)\vac\in M(1).
\end{align}
Then $\langle\omega\rangle\vac$ is the simple Virasoro vertex operator algebra
with central charge $1-6(p-1)^2/p$.
Since for $i\in\Z$
\begin{align}
\label{eq:un-type}
\omega_{i}u_{\ul{\zeta}}&=(\dfrac{1}{4p}\sum_{\begin{subarray}{l}j,k\in\N\\j+k=i-1\end{subarray}}
\alpha(j)\alpha(k)-
\dfrac{(p-1)i}{2p}\alpha(i-1))u_{\ul{\zeta}},
\end{align}
$u_{\ul{\zeta}}$ is a Whittaker vector of type 
\begin{align}
\label{eq:whittaker-vector-type}
\big(\dfrac{1}{4p}\sum_{\begin{subarray}{l}j,k\in\N\\j+k=i-1\end{subarray}}
\zeta_{j}\zeta_{k}-\delta_{i,r+1}\dfrac{(p-1)(r+1)}{2p}\zeta_{r}\big)_{i=r+1}^{2r+1}
\end{align}
for $\omega$.
If  $r\geq 1$ and $\zeta_{r}\neq 0$, then $\lw$ is an 
eigenvector for $\omega_{2r+1}$ with eigenvalue $\zeta_{r}^2/(4p)$
and hence $M(1,\ul{\zeta})$ is not an $\N$-graded weak $M(1)$-module by \eqref{eq:N-graded}.
We also note that the map 
\begin{align}
\label{eq:untwist-zeta-lambda}
\C^{r}\times \C^{\times}&\rightarrow \C^{r}\times \C^{\times}\nonumber\\
(\zeta_0,\ldots,\zeta_r)&\mapsto (\dfrac{1}{4p}\sum_{\begin{subarray}{l}j,k\in\N\\j+k=i-1\end{subarray}}
\zeta_{j}\zeta_{k}-\delta_{i,r+1}\dfrac{(p-1)(r+1)}{2p}\zeta_{r})_{i=r+1}^{2r+1}
\end{align}
is onto and
the images of $\ul{\zeta}=(\zeta_{i})_{i=0}^{r}, \zeta^{\prime}=(\zeta^{\prime}_{i})_{i=0}^{r}\in \C^{r}\times \C^{\times}$ under this map
are equal if and only if 
\begin{align}
\label{eq:zeta-equal}
\zeta_i+\zeta_i^{\prime}
&=\left\{
\begin{array}{ll}
0,&\mbox{for }i=1,\ldots,r,\mbox{ and }\\
2(p-1)(r+1)&\mbox{for }i=0
\end{array}
\right.
\end{align}
since
\begin{align}
&\dfrac{1}{4p}\sum_{\begin{subarray}{l}j,k\in\N\\j+k=i-1\end{subarray}}
\zeta_{j}\zeta_{k}-\delta_{i,r+1}\dfrac{(p-1)(r+1)}{2p}\zeta_{r}\nonumber\\
&=\left\{
\begin{array}{ll}
\dfrac{1}{4p}\zeta_{r}^2 &\mbox{if }i=2r+1,\\
\dfrac{1}{2p}\zeta_{r}\zeta_{i-1-r}+\dfrac{1}{4p}\sum\limits_{j=i-1-r+1}^{r-1}\zeta_{j}\zeta_{i-1-j}&\\
\quad{}-\delta_{i,r+1}\dfrac{(p-1)(r+1)}{2p}\zeta_{r} &\mbox{if }r+1\leq i<2r+1,
\end{array}
\right.
\end{align}
for $i=r+1,r+2,\ldots,2r+1$. In particular, this map is two-to-one.
Regarding the weak $M(1)$-module $M(1,\zeta)$ as a weak  $\langle\omega\rangle\vac$-module,
we have the following results:
\begin{lemma}
\label{lemma:generate-vir}
\begin{enumerate}
\item
For every $\zeta\in \C^{r}\times\C^{\times}$,
$M(1,\zeta)$ is a simple weak $\langle\omega\rangle\vac$-module.
\item 
For 
$\zeta=(\zeta_{i})_{i=0}^{r}\in \C^{r}\times \C^{\times}$ and 
$\zeta=(\zeta^{\prime}_{i})_{i=0}^{r^{\prime}}\in \C^{r^{\prime}}\times \C^{\times}$,
$
M(1,\ul{\zeta})\cong M(1,\ul{\zeta}^{\prime})$
as weak $\langle\omega\rangle\vac$-modules
if and only if $r=r^{\prime}$ and \eqref{eq:zeta-equal} holds.
\end{enumerate}
\end{lemma}
\begin{proof}
\begin{enumerate}
\item
For convenience, we define
\begin{align}
L(n)&=\omega_{n+1}
\end{align}
for  $n\in\Z$.
For $n\in\Z_{<0}$, we have
\begin{align}
\label{eq:recursion}
L(n+\wr)
&=\dfrac{1}{2p}
\alpha(n)\alpha(\wr)
+\dfrac{1}{4p}\sum_{\begin{subarray}{l}0\leq i\leq \wr-1,\\j\in\Z, i+j=n+r
\end{subarray}}
\alpha(j)\alpha(i)\nonumber\\
&\quad{}+\dfrac{1}{2p}\sum_{
\begin{subarray}{l}i\geq \wr+1,\\j\in\Z, i+j=n+r
\end{subarray}}
\alpha(j)\alpha(i)\nonumber\\
&\quad{}+\dfrac{1}{4p}
\sum_{\begin{subarray}{l}
0\leq j\leq \wr-1,\\
i<0,i+j=n+\wr
\end{subarray}}
\alpha(i)\alpha(j)+
\dfrac{1}{4p}\sum_{
\begin{subarray}{l}
i,j<0,\\
i+j=n+\wr
\end{subarray}}
\alpha(i)\alpha(j)\nonumber\\
&\quad{}+
\dfrac{p-1}{2p}(-n-\wr-1)\alpha(n+\wr)
\end{align} 
and therefore
\begin{align}
\label{eq:alpha-u}
L(n+\wr)u_{\zeta}
&=\dfrac{1}{2p}
\zeta_{\wr}\alpha(n)u_{\zeta}
+\dfrac{1}{4p}\sum_{i=0}^{\wr-1}
\zeta_{i}\alpha(n+\wr-i)u_{\zeta}\nonumber\\
&\quad{}+\dfrac{1}{4p}
\sum_{\begin{subarray}{l}
0\leq j\leq \wr-1,\\
n+\wr-j<0
\end{subarray}}
\zeta_{j}\alpha(n+\wr-j)u_{\zeta}+
\dfrac{1}{4p}\sum_{
\begin{subarray}{l}
i,j<0,\\
i+j=n+\wr
\end{subarray}}
\alpha(i)\alpha(j)u_{\zeta}\nonumber\\
&\quad{}+
\dfrac{p-1}{2p}(-n-\wr-1)\alpha(n+\wr)u_{\zeta}.
\end{align} 
Since $L(1)\alpha(-1)=(-2p+2)\vac$, for $i,j\in\Z$ we have
\begin{align}
\label{eq:L-alpha}
[L(j),\alpha(i)]
&=\sum_{k=0}^{\infty}\binom{j+1}{k}(L(k-1)\alpha(-1)\vac)_{i+j+1-k}\nonumber\\
&=-i\alpha(i+j)+(-2p+2)\binom{j+1}{2}\vac_{i+j-1,-1}.
\end{align} 
Thus, for $i\in\Z$, $m\in\Z_{\geq 0}$ and $j_1,\ldots,j_m\in\Z$, 
an inductive argument on $m$ shows that
\begin{align}
\label{eq:alphaLLL}
&\alpha(i)L(j_1)\cdots L(j_m)\nonumber\\
& \in \sum_{\wl\geq 0}\sum_{1\leq k_1<\cdots<k_{\wl}\leq m}
\C L(j_{k_1})\ldots L(j_{k_{\wl}})\nonumber\\
&\quad{}+\sum_{\wl\geq 0}\sum_{1\leq k_1<\cdots<k_{\wl}\leq m}
\sum_{\begin{subarray}{l}
T\subset\{1,\ldots,m\}\\
T\cap \{k_1,\ldots,k_{\wl}\}=\varnothing
\end{subarray}}
\C L(j_{k_1})\ldots L(j_{k_{\wl}})
\alpha(i+\sum_{t\in T}j_t).
\end{align}
For $j\in\Z$, we define 
\begin{align}
S_{j}&=\Span_{\C}\Big\{L(j_1)\cdots L(j_m)u_{\zeta}\ \Big|\ 
\begin{array}{l}
m\in\Z_{\geq 0},\\
j_1,\ldots,j_m\in\Z_{\leq r},\mbox{ and}\\
\mbox{for any }\varnothing\neq T\subset\{1,2,\ldots,m\}, \sum\limits_{t\in T}j_t\geq j
\end{array}\Big\}.
\end{align}
We note that $S_{j}\subset S_{j-1}$ for all $j\in\Z$.
We shall show that 
\begin{align}
\label{eq:alphauzeta}
\alpha(n)u_{\zeta}\in S_{n+r}
\end{align}
for all $n\in\Z$. If $n\geq 0$, then \eqref{eq:alphauzeta} follows from
\eqref{eq:alpha-uz-def}.
For $n<0$, we shall show \eqref{eq:alphauzeta} by induction on $n=0,-1,-2,\ldots$.
We have already shown that \eqref{eq:alphauzeta} holds for $n=0$.
Let $n\in\Z_{<0}$. Let $i,j\in\Z_{<0}$ with $i+j=n+r$. We have 
$\alpha(j)u_{\zeta}\in S_{j+r}$ and
$\alpha(i+k)u_{\zeta}\in S_{i+k+r}$ for $k\geq n-i+1$ by the induction hypothesis.
Since $j+r=n+2r-i\geq n+1+r$ and $i+j+r=n+2r\geq n+1+r$,
it follows from \eqref{eq:alphaLLL} that
\begin{align}
\alpha(i)\alpha(j)u_{\zeta}\in S_{n+1+r}.
\end{align}
Thus, \eqref{eq:alphauzeta} follows from \eqref{eq:alpha-u} and the induction hypothesis.
It follows from 
\eqref{eq:L-alpha} and \eqref{eq:alphauzeta} that $M(1,\zeta)=\langle\omega\rangle u_{\zeta}$
and therefore $M(1,\zeta)$ is simple by 
\eqref{eq:whittaker-vector-type},
\cite[Corollary 4.2]{OW} and \cite[Theorem 7]{LGZ}.

\item
It follows from (1) that $M(1,\zeta)=\langle\omega\rangle u_{\zeta}$.
Since any Whittaker vector of the type given in \eqref{eq:whittaker-vector-type} for $\omega$ in $M(1,\zeta)$ is a non-zero scalar multiple of  $\lw_{\zeta}$ 
by \cite[Proposition 3.2]{OW} and \cite[Theorem 2.3]{FJK},
(2) follows from 
\eqref{eq:whittaker-vector-type}--\eqref{eq:zeta-equal}.
\end{enumerate}
\end{proof}

Let 
\begin{align}
V_{\Z (\alpha/p)}\cong \bigoplus_{n\in\Z}M(1)\otimes_{\C}\C e^{n\alpha/p}
\end{align}
be the generalized vertex algebra associated to the lattice $\Z(\alpha/p)$ (See \cite{DL}).
We define a linear map
\begin{align}
\widetilde{Q}&=\Res_{x}Y(e^{-\alpha/p},x) : V_{\Z (\alpha/p)}\rightarrow V_{\Z (\alpha/p)}
\end{align}
and denote by $\sla{p}$ the kernel of $\widetilde{Q}$ in $M(1)$:
\begin{align}
\sla{p}&=\{a\in M(1)\ |\ \widetilde{Q}a=0\}.
\end{align}
For $i\in\N$, we also denote by $S_{i}(\alpha)$ the $i$-th coefficient of $t^i$ in the expansion
of $\exp(\sum_{j=1}^{\infty}\alpha(-j)t^j/j)$:
\begin{align}
\exp(\sum_{j=1}^{\infty}\dfrac{\alpha(-j)}{j}t^j)
&=\sum_{i=0}^{\infty}S_{i}(\alpha)t^i.
\end{align}
It is shown in \cite[Theorem 3.2]{A} that 
$\sla{p}$ is generated by $\omega$
and homogeneous
\begin{align}
H&=S_{2p-1}(\alpha)\vac\in M(1)
\end{align}
of weight $2p-1$.
A direct computation shows that
\begin{align}
\label{eq:H-truncate-C}
H_{i}\lw_{\zeta}&\in\C \lw_{\zeta}
\mbox{ for all }i\geq (r+1)(2p-1)-r-1
\end{align}
and, in particular, 
\begin{align}
\label{eq:H-truncate}
H_{i}\lw_{\zeta}&=
\left\{
\begin{array}{ll}
\dfrac{\zeta_{r}^{2p-1}}{(2p-1)!}\lw_{\zeta},& \mbox{if }i=(r+1)(2p-1)-1,\\
0,&\mbox{if }i>(r+1)(2p-1)-1.
\end{array}
\right.
\end{align}
It is shown in \cite[p.122]{A} that
$H$ is a primary vector of weight $2p-1$ for $\omega$, so
$\omega$ and $H$ satisfies
\begin{align}
\label{eq:lie-oj}
[\omega_{i},H_{j}]&=((2p-2)i-j)H_{i+j-1},\ i,j\in\Z.
\end{align}
The following result is a direct consequence of Lemma \ref{lemma:generate-vir} and \eqref{eq:H-truncate}:
\begin{corollary}
\label{corollray:Mzeta}
The weak $\sla{p}$-modules
$M(1,\zeta)$,$r\in\Z_{\geq 1}$ and $\zeta\in\C^{r}\times \C^{\times}$,
are simple and pairwise non-isomorphic.
In particular, for any odd integer $\ws\geq 3$,
 $\lambda=(\lambda_{(\ws+1)/2},\lambda_{(\ws+1)/2+1},\ldots,\lambda_{\ws})
\in \C^{(\ws-1)/2}\times\C^{\times}$, and
$q\in\C$ such that
\begin{align}
\label{eq:def-q-1}
q^2&=\dfrac{(4p)^{2p-1}}{((2p-1)!)^2}\lambda_{\ws}^{2p-1},
\end{align}
there exists a weak $\sla{p}$-module with a Whittaker vector $u$ of type $\ul{\lambda}$ for $\omega$ such that
\begin{align}
\label{eq:Hmlw-qlw}
H_{i}\lw_{\zeta}&=
\left\{
\begin{array}{rl}
q\lw& \mbox{for }i=(\ws+1)(2p-1)/2-1,\\
0&\mbox{for all }i>(\ws+1)(2p-1)/2-1.
\end{array}
\right.
\end{align}
\end{corollary}
\begin{proof}
It follows from Lemma \ref{lemma:generate-vir} (1) that 
$M(1,\zeta)$
is a simple weak $\sla{p}$-module for every $r\in\Z_{\geq 1}$ and $\zeta\in\C^{r}\times \C^{\times}$.
Let $\zeta\in\C^{r}\times \C^{\times}$ and 
$\zeta^{\prime}\in \C^{r^{\prime}}\times \C^{\times}$ such that 
$M(1,\zeta)\cong M(1,\zeta^{\prime})$ as weak $\sla{p}$-modules.
Since any Whittaker vector of the type given in \eqref{eq:whittaker-vector-type} for $\omega$ in $M(1,\zeta)=\langle\omega \rangle u_{\zeta}$ 
is a non-zero scalar multiple of  $\lw_{\zeta}$ by \cite[Proposition 3.2]{OW} and \cite[Theorem 2.3]{FJK}, Lemma \ref{lemma:generate-vir} (2) 
shows that $r=r^{\prime}$ and \eqref{eq:zeta-equal} holds. Therefore, $\zeta=\zeta^{\prime}$ by \eqref{eq:H-truncate}.
The equality \eqref{eq:Hmlw-qlw} follows from \eqref{eq:untwist-zeta-lambda} and \eqref{eq:H-truncate}.
\end{proof}

The following relation for $\omega$ and $H$ will be used in Lemma \ref{lemma:gen-omega}.
\begin{lemma}
\label{lemma:singular-vector}
\begin{align}
\label{eq:relation-wH}
\dfrac{3}{2p}H_{-4}\vac-\omega_{-1}H_{-2}\vac+\dfrac{2p-1}{2}\omega_{-2}H&=0.
\end{align}
\end{lemma}
\begin{proof}
For
\begin{align}
E(t)&=\exp(\sum_{i=1}^{\infty}\dfrac{\alpha(-i)}{i}t^{i}),
\end{align}
we have
\begin{align}
\label{eq:al-E}
\alpha(j)E(t)\vac&=\langle \alpha,\alpha\rangle t^jE(t)\vac
\mbox{ for }j\in\Z_{>0}.
\end{align}
Using \eqref{eq:al-E} repeatedly, we have
\begin{align}
\omega_{0}^{3}E(t)\vac
&=\sum_{i=1}^{\infty}(i+1)(i+2)\alpha(-i-3)t^iE(t)\vac\nonumber\\
&\quad{}+3
\sum_{i,j=1}^{\infty}(i+1)\alpha(-j-1)
\alpha(-2-i)t^{i+j}E(t)\vac\nonumber\\
&\quad{}+
\sum_{i,j,k=1}^{\infty}
\alpha(-1-k)\alpha(-1-j)\alpha(-1-i) t^{i+j+k}E(t)\vac,\nonumber\\
\omega_{-1}\omega_{0}E(t)\vac
&=\sum_{i=1}^{\infty}(i+1)\alpha(-3-i)t^iE(t)\vac\nonumber\\
&\quad{}+\sum_{i,j=1}^{\infty}
\alpha(-2-j)\alpha(-1-i)t^{i+j}E(t)\vac\nonumber\\
&\quad{}+\dfrac{1}{4p}\alpha(-1)^2\sum_{i=1}^{\infty}\alpha(-1-i)t^iE(t)\vac\nonumber\\
&\quad{}+\dfrac{p-1}{2p}
\alpha(-2)\sum_{i=1}^{\infty}\alpha(-1-i)t^iE(t)\vac,\nonumber\\
\omega_{-2}E(t)\vac&=\sum_{i=1}^{\infty}\alpha(-3-i)t^{i}E(t)\vac\nonumber\\
&\quad{}+\dfrac{1}{2p}\alpha(-1)\alpha(-2)E(t)\vac
+\dfrac{p-1}{p}\alpha(-3)E(t)\vac.
\end{align}
Thus
\begin{align}
\omega_0^3E(t)\vac&=
\dfrac{d^3}{dt^3}E(t)\vac-3\alpha(-1) \dfrac{d^2}{dt^2}E(t)\vac\nonumber\\
&\quad{}+(-3\alpha(-2)+3\alpha(-1)^2)\dfrac{d}{dt}E(t)\vac\nonumber\\
&\quad{}+(-2\alpha(-3)+3\alpha(-2)\alpha(-1)-\alpha(-1)^3)E(t)\vac,\nonumber\\
\omega_{-1}\omega_{0}E(t)\vac
&=\dfrac{d}{dt}(t^{-1}\dfrac{d}{dt}E(t))\vac-2\alpha(-1)t^{-1}\dfrac{d}{dt}E(t)\vac-\alpha(-2)\dfrac{d}{dt}E(t)\vac\nonumber\\
&\quad{}+\alpha(-1)E(t)t^{-2}\vac-\alpha(-3)E(t)\vac\nonumber\\
&\quad{}+\alpha(-1)^2t^{-1}E(t)\vac+\alpha(-1)\alpha(-2)E(t)\vac\nonumber\\
&\quad{}
+\dfrac{1}{4p}\alpha(-1)^2(\dfrac{d}{dt}E(t)-\alpha(-1)E(t))\vac
+\dfrac{p-1}{2p}\alpha(-2)(\dfrac{d}{dt}E(t)-\alpha(-1)E(t))\vac,\nonumber\\
\omega_{-2}E(t)\vac
&=t^{-2}\dfrac{d}{dt}E(t)\vac-\alpha(-1)t^{-2}E(t)\vac-\alpha(-2)t^{-1}E(t)\vac-\alpha(-3)E(t)\vac\nonumber\\
&\quad{}+\dfrac{1}{2p}\alpha(-1)\alpha(-2)E(t)\vac+\dfrac{p-1}{p}\alpha(-3)E(t)\vac
\end{align}
and therefore
\begin{align}
\omega_{0}^3H&=
(2p)(2p+1)(2p+2)S_{2p+2}(\alpha)\vac-6p(2p+1)\alpha(-1)S_{2p+1}(\alpha)\vac\nonumber\\
&\quad{}+(-3\alpha(-2)+3\alpha(-1)^2)(2p)S_{2p}(\alpha)\vac\nonumber\\
&\quad{}+(-2\alpha(-3)+3\alpha(-2)\alpha(-1)-\alpha(-1)^3)S_{2p-1}(\alpha)\vac,\nonumber\\
\omega_{-1}\omega_{0}H&=(2p) (2p+2)S_{2p+2}(\alpha)\vac +(-4p-1)\alpha(-1)S_{2p+1}(\alpha)\vac \nonumber\\
&\quad{}+\dfrac{3}{2}\alpha(-1)^2 S_{2p}(\alpha)\vac+(-p - 1)\alpha(-2)S_{2p}(\alpha)\vac\nonumber \\
&\quad{}-\alpha(-3)S_{2p-1}(\alpha)\vac- \dfrac{1}{4 p}\alpha(-1)^3S_{2p-1}(\alpha)\vac + \dfrac{p + 1}{2 p} \alpha(-2) \alpha(-1)S_{2p-1}(\alpha)\vac,
\nonumber\\
\omega_{-2}H&=(2p + 2)S_{2p+2}(\alpha)\vac - \alpha(-1)S_{2p+1}(\alpha)\vac- \alpha(-2)S_{2p}(\alpha)\vac\nonumber\\
&\quad{}-\alpha(-3)S_{2p-1}(\alpha)\vac + \dfrac{1}{2p}\alpha(-1)\alpha(-2)S_{2p-1}(\alpha)\vac + 
\dfrac{p - 1}{p}\alpha(-3)S_{2p-1}(\alpha)\vac.
\end{align}
Since $\omega_0^iH=i!H_{-1-i}\vac$ for $i\in\Z_{\geq 0}$, the assertion follows from a direct computation.
\end{proof}

The same argument as in the proof of \cite[Lemma 3.2]{T1} shows 
the following result:
\begin{lemma}
\label{lemma:top}
Let $\ws\in\Z$ with $\ws\geq 2$.
Let $M$ be a weak $\langle\omega\rangle\vac$-module and $\lw\in M$ such that
$\omega_{i}\lw=0$ for all $i>\ws$.
Let
$p,q,i_1,\ldots,i_{p},i_{p+1},\ldots,i_{q}\in\Z$
such that $0\leq p\leq q$, $i_1,\ldots,i_{p}\geq 0,i_{p+1},\ldots,i_{q}<0$, and
	$\ws \wq\leq \wi_1+\cdots+\wi_{q}$. If $p<q$ or $(i_1,\ldots,i_{p})\neq(\ws,\ldots,\ws)$, then
	\begin{align}
\label{eq:oo}
		\omega_{i_q}\cdots \omega_{i_{p+1}}\omega_{i_p}\cdots \omega_{i_1}\lw&=0.
	\end{align}
\end{lemma}

The following results will be used in Lemmas \ref{lemma:re-HH}.
\begin{lemma}
\label{lemma:formula-omega}
Let $\ws\in\Z$ with $\ws\geq 2$, $M$ 
a weak $\sla{p}$-module, and $\lw\in M$ such that
$\omega_{i}\lw=0$ for all $i>\ws$.
Let $q\in\Z_{\geq 0}$ and $j_1,\ldots,j_q\in\Z_{\leq -1}$, and
define
\begin{align}
\label{eq:qj2k}
k&=\wt (\omega_{j_1}\cdots\omega_{j_q}\vac)=-j_1-\cdots-j_q+q.
\end{align}
Let $m\in\Z_{\geq(\ws+1)k/2-1}$.
\begin{enumerate}
\item If $k\neq 2q$ or $m>(\ws+1)k/2-1$, then
\begin{align}
\label{eq:lem-omega-0}
(\omega_{j_1}\cdots\omega_{j_q})_{m}\lw&=0.
\end{align}
\item If $k=2q$, then
\begin{align}
j_1=\cdots=j_{k/2}&=-1
\end{align}
and
\begin{align}
\label{eq:omega-action-zeta}
(\omega_{-1}^{k/2})_{(\ws+1)k/2-1}\lw&
=\omega_{\ws}^{k/2}\lw.
\end{align}
\item
For $i,j\in\Z$ with $i+j=(\ws+1)(2p-1)-2$, we have
\begin{align}
[H_{i},H_{j}]u&=0.
\end{align}
\end{enumerate}
\end{lemma}
\begin{proof}
It follows from \eqref{eq:qj2k} that $2q\leq k$.
Thus, for $i_1,\ldots,i_{q}\in\Z$
 such that
$i_1+\cdots+i_q=j_1+\cdots+j_q+m+1$, we have
\begin{align}
\label{eq:k-q}
i_1+\cdots+i_q&
=q-k+m+1\nonumber\\
&\geq q-k+(\ws+1)k/2-1+1\nonumber\\
&=\dfrac{(\ws-1)\wk}{2}+q\geq \ws\wq.
\end{align}
It follows from Lemma \ref{lemma:top} that
if $(i_1,\ldots,i_q)\neq (\ws,\ldots,\ws)$, then
\begin{align}
\label{eq:omega-0}
\nor \omega_{i_1}\cdots\omega_{i_q}\nor\lw&=0.
\end{align}
Here, for $i_1,\ldots,i_{q}\in\Z$ we define 
$\nor \omega_{i_1}\cdots\omega_{i_{q}}\nor$ inductively by
\begin{align}
\label{eq:nomal-ordering}
\nor \omega_{i_1}\nor&=\omega_{i_1}\qquad\mbox{ and}\nonumber\\
\nor \omega_{i_1}\cdots\omega_{i_{q}}\nor&=
\left\{
\begin{array}{ll}
\nor \omega_{i_2}\cdots\omega_{i_{q}}\nor \omega_{i_1}&\mbox{if }i_1\geq 0,\\
\omega_{i_1}\nor \omega_{i_2}\cdots\omega_{i_{q}}\nor &\mbox{if }i_1<0.
\end{array}\right.
\end{align}
\begin{enumerate}
\item
If $k>2q$ or $m>(\ws+1)k/2-1$, then
\begin{align}
i_1+\cdots+i_q&>\ws\wq
\end{align}
by \eqref{eq:k-q}. Thus, \eqref{eq:lem-omega-0} follows from \eqref{eq:omega-0}.
\item
Suppose $k=2q$. It follows from \eqref{eq:qj2k} that $j_1=\cdots=j_{k}=-1$.
It follows from \eqref{eq:k-q} and \eqref{eq:omega-0} that
\begin{align}
(\omega_{-1}^{k/2}\vac)_{(\ws+1)k/2-1}\lw
&=\sum_{
\begin{subarray}{l}
i_1,\ldots,i_{k/2}\in\Z\\
i_1+\cdots+i_{k/2}+k/2=(\ws+1)k/2
\end{subarray}
}\nor\omega_{i_1}\cdots\omega_{i_{k/2}}\nor\lw\nonumber\\
&=\omega_{\ws}^{k/2}\lw.
\end{align}
\item
It is shown in \cite[Lemma 3.2]{A} that $H_{k}H\in\langle\omega \rangle\vac$
 for all $k\geq -2p$.
Thus, for $k=0,1,\ldots$ since
\begin{align}
&(\ws+1)(2p-1)-2-k-(\dfrac{\ws+1}{2}\wt(H_{k}H)-1)\nonumber\\
&=\dfrac{\ws-1}{2}(\wk+1)>0,
\end{align}
we have
\begin{align}
(H_{k}H)_{(\ws+1)(2p-1)-2-k}&=0
\end{align}
by (1) and therefore
\begin{align}
[H_i,H_j]&=\sum_{k=0}^{\infty}
\binom{i}{k}(H_{k}H)_{(\ws+1)(2p-1)-2-k}\lw=0.
\end{align}
\end{enumerate}
\end{proof}

By Lemma \ref{lemma:formula-omega}, we have the following relation for $\omega$ and $H$,
which will be used in Lemma \ref{lemma:gen-omega} together with \eqref{eq:relation-wH}.
\begin{lemma}
\label{lemma:re-HH}
\begin{align}
\label{eqref:H-1H}
H_{-1}H&=\dfrac{(4p)^{2p-1}}{((2p-1)!)^2}\omega_{-1}^{2p-1}\vac+a
\end{align}
where $a$ is an element of 
\begin{align}
\label{eq:neq(-1,-1)}
\Span_{\C}\Big\{\omega_{j_1}\cdots \omega_{j_k}\vac\ \Big|\ 
\begin{array}{l}
k\in\Z_{\geq 0}, j_1,\ldots,j_k\in\Z_{<0},\\
\wt (\omega_{j_1}\cdots \omega_{j_k})=4p-2,\\
(j_1,\ldots,j_k)\neq (-1,\ldots,-1)
\end{array}
\Big\}.
\end{align}
\end{lemma}
\begin{proof}
By \cite[Lemma 3.2]{A}, we can write 
\begin{align}
H_{-1}H&=c \omega_{-1}^{2p-1}\vac+ a
\end{align}
where $c\in\C$ and $a$ is an element of \eqref{eq:neq(-1,-1)}.
By \eqref{eq:H-truncate} and Lemma \ref{lemma:formula-omega} (3) with 
$\ws=2\wr+1$, we have
\begin{align}
\label{eq:Hm1H}
&(H_{-1}H)_{2(r+1)(2p-1)-1}\lw_{\zeta}\nonumber\\
&=
\sum_{\begin{subarray}{l}
i+j=2(r+1)(2p-1)-2,\\
i\geq 0
\end{subarray}}H_{j}H_{i}\lw_{\zeta}
+\sum_{\begin{subarray}{l}i+j=2(r+1)(2p-1)-2,\\
i<0
\end{subarray}}H_{i}H_{j}\lw_{\zeta}\nonumber\\
&=
\sum_{\begin{subarray}{c}i+j=2(r+1)(2p-1)-2,\\0\leq i\leq 
(\wor+1)(2p-1)-1\end{subarray}}([H_{j},H_{i}]+H_{i}H_{j})\lw_{\zeta}\nonumber\\
&=H_{(r+1)(2p-1)-1}^2\lw_{\zeta}.
\end{align}
By Lemma \ref{lemma:formula-omega} (1), we also have
\begin{align}
a_{2(r+1)(2p-1)-1}u_{\zeta}&=0.
\end{align}
It follows from \eqref{eq:un-type} and Lemma \ref{lemma:formula-omega} (2) that
\begin{align}
(\omega_{-1}^{2p-1}\vac)_{2(\wr+1)(2p-1)-1}u_{\zeta}&=(\omega_{2\wr+1})^{2p-1}u_{\zeta}=
\dfrac{\zeta_{r}^{4p-2}}{(4p)^{2p-1}}u_{\zeta}.
\end{align}
By \eqref{eq:H-truncate} and \eqref{eq:Hm1H},
we  have
$c=(4p)^{2p-1}/((2p-1)!)^2$.
\end{proof}
\begin{remark}
Lemma \ref{lemma:re-HH} also follows from the proof of \cite[Lemma 6.1]{A}.
\end{remark}

\section{\label{section:main}Weak $\sla{p}$-modules with Whittaker vectors.}
In this section, we will show Theorem \ref{theorem:classification-module}.
The following  is a key result to show Theorem \ref{theorem:classification-module}.
\begin{lemma}
\label{lemma:gen-omega}
Let $\ws, \wot\in\Z$ with $2\leq \wot\leq \ws$, $\ul{\lambda}=(\lambda_{t},\ldots,\lambda_{\ws})\in \C^{\ws-\wot+1}$
with $\lambda_{\ws}\neq 0$, and $q\in \C$ such that
\begin{align}
\label{eq:def-q}
q^2&=\dfrac{(4p)^{2p-1}}{((2p-1)!)^2}\lambda_{\ws}^{2p-1},
\end{align}
and define
\begin{align}
\label{eq:define-m}
\wm&=\dfrac{(\ws+1)(2\wp-1)}{2}-1.
\end{align}
Let $M$ be a weak $\sla{p}$-module and 
$\lw\in M$ a non-zero element such that
$\omega_{i}\lw=\lambda_{i}\lw$ 
for $i=\wot,\ldots,\ws$ and
$\omega_{i}\lw=0$ for all $i>\ws$.
\begin{enumerate}
\item 
The integer $\ws$ is odd and
\begin{align}
\label{eq:hm2}
H_{m}^2\lw&=\dfrac{(4p)^{2p-1}}{((2p-1)!)^2}\lambda_{\ws}^{2p-1}\lw\nonumber\\
H_{i}\lw&=0\mbox{  for all  }\wi>\wm.
\end{align}
Moreover, the submodule of $M$ generated by $\lw$ is equal to $\langle \omega\rangle (\C\lw+\C H_{m}\lw)$.
\item
If $H_{m}\lw \not\in \C\lw$, then two nonzero elements $v_{\pm}=H_{\wm}\lw\pm q\lw\in M$ satisfy
\begin{align}
\label{eq:omega-t-s}
\omega_{i}\wv_{\pm}
&=\left\{
\begin{array}{rl}
\lambda_{i}\wv_{\pm}&\mbox{for }\wot\leq i\leq \ws,\\
0&\mbox{for all }\wi>\ws
\end{array}\right.
\end{align}
and
\begin{align}
\label{eq:hmv}
H_{m}\wv_{\pm}=&\pm q\wv_{\pm}.
\end{align}
\item Suppose that $H_{m}u\in\C u$.
Then
\begin{align}
\label{eq:hmupmu}
H_{m}u=qu\mbox{ or }-qu
\end{align}
and
\begin{align}
H_{i}u&=\mu_iu \mbox{ for }\wm-\ws+\wot\leq \wi\leq \wm \label{eq:Jw-3}
\end{align}
where $\mu_i$ is a polynomial in $\lambda_{\ws-\wm+i},\lambda_{\ws-\wm+i+1},\ldots,\lambda_{\ws-1},\lambda_{\ws}^{-1}$ and $q$ for each $i=\wm-\ws+\wot,\wm-\ws+\wot+1,\ldots, \wm$.
Moreover, the submodule of $M$ generated by $u$ is equal to $\langle\omega\rangle u$.

\item If $t=(s+1)/2$ and  $H_{m}u\in \C u$, then the submodule of $M$ generated by $u$
is simple.
\item If $t=(s+1)/2$ and $M$ is simple, then $H_{m}u\in\C u$.
\end{enumerate}
\end{lemma}

\begin{proof}
\begin{enumerate}
\item
For $n\in\Z$, by \eqref{eq:relation-wH} we have
\begin{align}
\label{eq:oH-induction}
0
&=(\dfrac{2p}{3}H_{-4}\vac-\omega_{-1}H_{-2}\vac+\dfrac{2p-1}{2}\omega_{-2}H)_{\wn+\ws+2}u\nonumber\\
&=\dfrac{2p}{3}\binom{-n-\ws}{3}H_{\wn+\ws-1}u
-\sum_{
\begin{subarray}{l}i,j\in\Z\\ i+j=\wn+\ws\end{subarray}}(-j-1)\nor \omega_{i}H_{j}\nor u\nonumber\\
&\quad{}+\dfrac{2p-1}{2}\sum_{
\begin{subarray}{l}i,j\in\Z\\ i+j=\wn+\ws\end{subarray}}(-i-1)\nor\omega_{i}H_{j}\nor u\nonumber\\
&\equiv
-(-n-1)H_{\wn}\lambda_{\ws}u+\dfrac{2p-1}{2}(-\ws-1)H_{\wn}\lambda_{\ws}u
\pmod{\sum_{j\geq n+1}\langle \omega \rangle H_{j}u}.\nonumber\\
&=\lambda_{\ws}(n+1-\dfrac{(\ws+1)(2\wp-1)}{2})H_{\wn}u.
\end{align}
Here $\nor \omega_{i}H_{j}\nor$ is defined to be
\begin{align}
\nor \omega_{i}H_{j}\nor&=
\left\{
\begin{array}{ll}
H_{j}\omega_{i}&\mbox{if }i\geq 0,\\
\omega_{i}H_{j}&\mbox{if }i<0.
\end{array}
\right.
\end{align}
Since $\sla{p}$ is simple by \cite[Theorem 4.3]{AM1}, it follows from \cite[Proposition 11.9]{DL} that
there exists $\wm^{\prime}\in\Z$ such that $H_{\wm^{\prime}}\lw\neq 0$ and $H_{\wj}\lw=0$ for all $\wj>\wm^{\prime}$.
It follows from \eqref{eq:oH-induction} that $\ws$ is an odd integer and $\wm^{\prime}=(\ws+1)(2\wp-1)/2-1=m$.
By \eqref{eq:lie-oj} and \eqref{eq:oH-induction}, an inductive argument shows that
\begin{align}
\label{eq:Jinomega}
H_{n}\lw\in \langle \omega \rangle(\C\lw+\C H_{m}\lw)\mbox{ for all }n\in\Z 
\end{align}
and therefore 
the submodule of $M$ generated by $\lw$ is equal to $\langle \omega\rangle (\C\lw+\C H_{m}\lw)$.
By using Lemma \ref{lemma:re-HH}, the same argument as in the proof of Lemma \ref{lemma:re-HH}
shows that 
\begin{align}
\label{eq:hm2u}
H_{\wm}^2\lw&=\dfrac{(4p)^{2p-1}}{((2p-1)!)^2}\lambda_{s}^{2p-1}\lw.
\end{align}
\item The assertion follows from
\eqref{eq:lie-oj} and \eqref{eq:hm2}.
\item
Equation  \eqref{eq:hmupmu} follows from \eqref{eq:hm2}.
By \eqref{eq:lie-oj} and \eqref{eq:relation-wH}, we have
\begin{align}
0
&=(\dfrac{3}{2p}H_{-4}\vac-\omega_{-1}H_{-2}\vac+\dfrac{2p-1}{2}\omega_{-2}H)_{\wot+\wm+2}u\nonumber\\
&=\dfrac{3}{2p}
\binom{-\wm-\wot}{3}H_{\wot+\wm-1}u
-
\sum_{
\begin{subarray}{l}\wot\leq i\leq \ws\end{subarray}}(-(\wot+\wm-i)-1)H_{\wot+\wm-i}\omega_{i} u\nonumber\\
&\quad{}+\dfrac{2p-1}{2}\sum_{
\begin{subarray}{l}\wot\leq i\leq \ws\end{subarray}}(-i-1)H_{\wot+\wm-i}\omega_{i} u\nonumber\\
&=\sum_{
\begin{subarray}{l}\wot\leq i\leq \ws\end{subarray}}(\wot+\wm-i+1-\dfrac{(i+1)(2\wp-1)}{2})H_{\wot+\wm-i}\omega_{i}u\nonumber\\
&=(\wot+\wm-\ws+1-\dfrac{(\ws+1)(2\wp-1)}{2})H_{\wot+\wm-\ws}\omega_{\ws}u\nonumber\\
&\quad{}+\sum_{i=\wot}^{\ws-1}(\wot+\wm-i+1-\dfrac{(i+1)(2\wp-1)}{2})H_{\wot+\wm-i}\omega_{i}u\nonumber\\
&=(\wot-\ws)H_{\wot+\wm-\ws}\lambda_{\ws}u\nonumber\\
&\quad{}+\sum_{
\begin{subarray}{l}i=\wot\end{subarray}}^{\ws-1}(\wot+\wm-i+1-\dfrac{(i+1)(2\wp-1)}{2})H_{\wot+\wm-i}\lambda_{i}u.
\end{align}
Thus, an inductive argument on $\wot=\ws,\ws-1,\ldots$ shows \eqref{eq:Jw-3}.
It follows from (1) that
the submodule of $M$ generated by $u$ is equal to $\langle \omega\rangle u$.

\item Since $u$ is a Whittaker vector of type $(\lambda_i)_{i=(s+1)/2}^{s}$ for $\omega$,
the assertion follows from 
(3), \cite[Corollary 4.2]{OW}, and \cite[Theorem 7]{LGZ}.
\item 
Suppose  $H_{m}u\not\in \C$. Let $v_{\pm}\in M$ be two nonzero vectors defined in (2).
We note that $v_{\pm}$ are Whittaker vectors for $\omega$.
Since $M$ is simple we have $M=\langle\omega \rangle v_{+}=\langle\omega \rangle v_{-}$ by (3).
Since $v_{\pm}$ are Whittaker vectors of type $(\lambda_i)_{i=(\ws+1)/2}^{\ws}$ for $\omega$ in $M$,
$v_{+}$ is a non-zero scalar multiple of  $v_{-}$ by \cite[Proposition 3.2]{OW} and \cite[Theorem 2.3]{FJK}.
This contradicts to \eqref{eq:hmv}.
\end{enumerate}
\end{proof}

Now we give a proof of Theorem \ref{theorem:classification-module}.

\begin{proof}[(Proof of Theorem \ref{theorem:classification-module})]
The following argument is almost the same as in the proof of \cite[Theorem 1.1]{T1}.
Let $\ws$ be an odd integer with $\ws\geq 2$,
$\lambda=(\lambda_{(\ws+1)/2},\lambda_{(\ws+1)/2+1},\ldots,\lambda_{\ws})
\in \C^{(\ws-1)/2}\times\C^{\times}$, $m\in\Z$ defined by \eqref{eq:define-m}, and $q\in\C$ which satisfies \eqref{eq:def-q-1}. 
Taking a quotient space of the tensor algebra of $\sla{p}\otimes \C[t,t^{-1}]$ by the two sided ideal
generated by the Borcherds identity, 
$H\otimes t^i\ (i>\wm)$, $H\otimes t^{m}-q(\vac \otimes 1)$, $\omega\otimes t^{i}\ (i>\ws)$,
and $\omega\otimes t^{i}-\lambda_i(\vac \otimes 1)\ ((\ws+1)/2\leq i\leq \ws)$, we obtain a pair $(N(\lambda,q),\ulw(\lambda,q))$ of a weak $\sla{p}$-module 
$N(\lambda,q)$ and a Whittaker vector $\ulw(\lambda,q)\in N(\lambda,q)$ of type $\ul{\lambda}$ for $\omega$ such that
$Y_{N(\lambda,q)}(H,x)\ulw(\lambda,q)\in x^{-\wm-1}N(\lambda,q)[[x]]$ and $H_m\ulw(\lambda,q)=q\ulw(\lambda,q)$ with the following universal property:
for any pair $(U,u)$ of a weak $\sla{p}$-module $U$ and a Whittaker vector $u\in U$ of type $\ul{\lambda}$
such that $Y_{U}(H,x)u\in x^{-\wm-1}U[[x]]$ and $H_{m}u=qu$,
there exists a unique weak $\sla{p}$-module homomorphism $N(\lambda,q)\rightarrow U$ which maps $\ulw(\lambda,q)$ to $u$.
The weak module $N(\lambda,q)$ is not zero by Corollary \ref{corollray:Mzeta}
and simple  by Lemma \ref{lemma:gen-omega} (4).

Let $\module$ be a simple weak $\sla{p}$-module with a Whittaker vector $\lw$
of type $\lambda$ for $\omega$.
By Lemma \ref{lemma:gen-omega} (3) and (5), $\module$ is isomorphic to a quotient weak module of
$N(\lambda,q)$ or $N(\lambda,-q)$ and, moreover, $N(\lambda,q)\cong M$ or $N(\lambda,-q)\cong M$ since $N(\lambda,\pm q)$ are simple.
Thus,  $M$ is isomorphic to one of the weak $\sla{p}$-modules listed in Theorem \ref{theorem:classification-module}
by Corollary \ref{corollray:Mzeta}.
The proof is complete.
\end{proof}


\begin{thebibliography}{99}
\bibitem{A}
D. Adamovi\'c, 
Classification of irreducible modules of certain subalgebras of free boson vertex algebra,
{\em J. Algebra} {\bfseries 270} (2003), 115--132. 
\bibitem{ALZ}
D. Adamovi\'c, R. L\"u, and. K. Zhao, 
Whittaker modules for the affine Lie algebra $A_{1}^{(1)}$,
{\em Adv. Math.} {\bfseries 289} (2016), 438--479. 
\bibitem{AM1}
D. Adamovi\'c and A. Milas,
Logarithmic intertwining operators and ${\mathcal W}(2,2p-1)$ algebras,
{\em J. Math. Phys.} {\bfseries 48} (2007), 073503
\bibitem{AM2}
D. Adamovi\'c and A. Milas,
On the triplet vertex algebra ${\sW}(p)$,
{\em Adv. Math.} {\bfseries 217} (2008), 2664--2699.
\bibitem{AP}
D. Arnal and G. Pinczon,
On algebraically irreducible representations of the Lie algebra $sl(2)$,
{\em J. Math. Phys.} {\bfseries 15} (1974), 350--359.
\bibitem{B}
R. Borcherds, 
Vertex algebras, Kac-Moody algebras, and the Monster, 
{\em Proc. Nat. Acad. Sci. U.S.A.} {\bfseries 83} (1986), 3068--3071.
\bibitem{BCNM}
P. Bouwknegt, A. Ceresole, P. van Nieuwenhuizen, and J. McCarthy, 
Extended Sugawara construction for the superalgebras ${\rm SU}(M+1|N+1)$. II. The third-order Casimir algebra,
{\em Phys. Rev. D} {\bfseries 40} (1989), 415--421.
\bibitem{BM}
K. Bringmann and A. Milas, 
${\mathcal W}$-algebras, false theta functions and quantum modular forms, I,
I.M.R.N. {\bfseries 21} (2015), 11351--11387. 
\bibitem{CM}
T. Creutzig and A. Milas,
The false theta functions and the Verlinde formula,
{\em  Adv. Math.} {\bfseries 262} (2014) 520--545.
\bibitem{CMP}
T. Creutzig, A. Milas, and M. Rupert,
Logarithmic link invariants of $\overline{U}^{H}_{q}({\mathfrak sl}_2)$ and asymptotic dimensions of singlet vertex algebras,
 arXiv:1605.05634.
\bibitem{DL}
C. Dong and J. Lepowsky, \emph{Generalized vertex algebras and relative vertex operators}, 
Progress in Mathematics {\bfseries 112}, Birkhauser Boston, Inc., Boston, MA, 1993.
\bibitem{E}
P. Etingof, 
Whittaker functions on quantum groups and $q$-deformed Toda operators,
in \emph{Differential Topology, Infinite-Dimensional Lie Algebras, and Applications}, 
Amer. Math. Soc. Transl. Ser. 2 {\bfseries 194}, Amer. Math. Soc., Providence, 1999, 9--25.
\bibitem{FJK}
E. Feli\'nska, Z. Jask\'olski, and M. Koszto\l owicz,
Whittaker pairs for the Virasoro algebra and the Gaiotto-Bonelli-Maruyoshi-Tanzini states,
{\em J. Math. Phys.} {\bfseries 53} (2012), no. 3, 033504.\\
Erratum,  {\em J. Math. Phys.} {\bfseries 53} (2012), no. 12, 129902.
\bibitem{FLM}
I. B. Frenkel, J. Lepowsky and A. Meurman, 
\emph{Vertex Operator Algebras and the Monster}, Pure and Applied Math., Vol. {\bfseries 134}, Academic Press, 1988.
\bibitem{G}
D. Gaiotto,
Asymptotically free $\sN = 2$ theories and irregular conformal blocks,
Journal of Physics: Conference Series {\bfseries 462} (2013) 012014
\bibitem{Ka}
H. G. Kausch, 
Extended conformal algebras generated by a multiplet of primary fields,
{\em Phys. Lett. B} {\bfseries 259} (1991), 448--455. 
\bibitem{K1}
B. Kostant, 
On Whittaker vectors and representation theory,
{\em Invent. Math.} {\bfseries 48} (1978), 101--184. 
\bibitem{K2}
B. Kostant, 
The solution to a generalized Toda lattice and representation theory,
{\em  Adv. Math.} {\bfseries 34 } (1979), 195--338. 
\bibitem{LL}
J. Lepowsky and H. S. Li, 
\emph{Introduction to Vertex Operator
Algebras and their Representations}, Progress in Mathematics \textbf{227}, Birkhauser Boston, Inc., Boston, MA, 2004.
\bibitem{Li1}
H. S. Li,
Local systems of vertex operators, vertex superalgebras and modules,
\emph{J. Pure Appl. Algebra} {\bfseries 109} (1996), 143--195. 
\bibitem{LGZ}
R. L\"u, X. Guo, K. Zhao, Irreducible modules over the Virasoro algebra,
{\em Doc. Math.} {\bfseries 16} (2011), 709--721.
\bibitem{OW}
M. Ondrus and E. Wiesner, 
Whittaker modules for the Virasoro algebra,
{\em J. Algebra Appl.} {\bfseries 8} (2009), 363--377. 
\bibitem{S}
A. Sevostyanov, 
Quantum deformation of Whittaker modules and the Toda lattice,
{\em Duke Math. J.} {\bfseries 105} (2000), 211--238. 
\bibitem{T1}
K. Tanabe,
Simple weak modules for the
fixed point subalgebra of the Heisenberg vertex operator algebra of rank $1$ by an automorphism of order $2$ and Whittaker vectors, 
to appear in Proceedings of the AMS.
\bibitem{W1}
W. Wang,
$W_{1+\infty}$ algebra, $W_3$ algebra, and Friedan-Martinec-Shenker bosonization. 
{\em Comm. Math. Phys.} {\bfseries 195} (1998), 95--111. 
\bibitem{W2}
W. Wang, 
Classification of irreducible modules of $W_3$ algebra with $c=-2$,
{\em Comm. Math. Phys.} {\bfseries 195} (1998), 113--128.
\bibitem{Z}
Y. Zhu,
Modular invariance of characters of vertex operator algebras,
{\em J. Amer. Math. Soc.} {\bf9} (1996), 237--302.

\end{thebibliography}
\end{document}